\documentclass[psamsfonts]{amsart}
\usepackage{latexsym, amssymb, amsmath, mathrsfs}

\newtheorem{theorem}{Theorem}[section]

\theoremstyle{definition}
\newtheorem{definition}[theorem]{Definition}
\newtheorem{corollary}[theorem]{Corollary}

\newtheorem{question}[theorem]{Question}
\newtheorem{notation}[theorem]{Notation}
\numberwithin{equation}{section}
\DeclareMathOperator{\sgn}{\alpha}
\DeclareMathOperator{\cyd}{\gamma}
\DeclareMathOperator{\clf}{\phi}
\DeclareMathOperator{\hdm}{\eta}

\newcommand{\bracket}[1]{[\,#1\,]}
\newcommand{\norm}[1]{\|\,#1\,\|}
\newcommand{\ip}[2]{\left\langle#1,#2\right\rangle}
\newcommand{\conj}[1]{{#1}^*} 
\newcommand{\Sum}[1]{\sum_{#1}}
\newcommand{\Z}[1]{\mathbb{Z}_2^#1}
\newcommand{\cds}[1]{\mathbb{S}_{#1}}
\newcommand{\sob}[1]{{\langle#1\rangle}}
\newcommand{\sbf}[1]{(-1)^{\langle#1\rangle}}                               
\newcommand{\tga}[3]{\left[#1,#2,#3\right]}
\newcommand{\tgar}[2]{\left[#1,#2\right]}
\newcommand{\basis}[1]{\mathcal{#1}}
\newcommand{\twisteq}[2]{#1\vee#2}
\newcommand{\twistne}[2]{#1\wedge#2}

\makeindex
\begin{document}

\title{Properly Twisted Groups and their Algebras}

\author{ John W. Bales (2006)}
\address{Department of Mathematics\\
         Tuskegee University\\
	 Tuskegee, AL 36088\\
	 USA}
\email{jbales@tuskegee.edu}

\date{}

\begin{abstract}
A proper twist on a group $G$ is a function $\sgn:G\times G\mapsto\{-1,1\}$ with the property that, if $p,q\in G$ then
$\sgn(p,q)\sgn(q,q^{-1})=\sgn(pq,q^{-1})$ and $\sgn(p^{-1},p)\sgn(p,q)=\sgn(p^{-1},pq).$ The span $V$ of a set of unit
vectors $\basis{B}=\{i_p\,|\,p\in G\}$ over a ring $\mathbb{K}$ with product $xy=\Sum{p,q}\sgn(p,q)x_py_qi_{pq}$ is a
twisted group algebra. If the twist $\sgn$ is proper, then the conjugate defined by
$\conj{x}=\Sum{p}\sgn(p,p^{-1})\conj{x}_pi_{p^{-1}}$ and inner product $\ip{x}{y}=\Sum{p}x_p\conj{y}_p,$ satisfy the
adjoint properties $\ip{xy}{z}=\ip{y}{\conj{x}z}$ and $\ip{x}{yz}=\ip{x\conj{z}}{y}$ for all $x,y,z\in V$. Proper twists
on $\Z{N}$ over the reals produce the complex numbers, quaternions, octonions and all higher order Cayley-Dickson and
Clifford algebras.
\end{abstract} 
\maketitle
\section{Introduction}

A vector space $V$ with orthonormal basis $\basis{B}$ can be transformed into an algebra by defining a product for the elements of $\basis{B}$ and extending that product to $V.$ An obvious way to define a product on $\basis{B}$ is to associate each element of $\basis{B}$ with a corresponding element of a group $G$ so that given $p,q\in G$ and $i_p,i_q\in \basis{B}$, one has the product $i_pi_q=i_{pq}$. Extending this product to $V$ transforms $V$ into a \emph{group algebra}. This scheme may be modified by adding a `twist'. Define a function $\alpha: G\times G\mapsto \{-1,1\}$ and define the product $i_pi_q=\alpha(p,q)i_{pq}.$ Extending this product to $V$ transforms $V$ into a \emph{twisted} group algebra \cite{BS1970}. Specific properties of the twist give rise to specific properties of the algebra. For each non-negative integer $N$, the Euclidean vector spaces of dimension $2^N$ may be transformed into either a Cayley-Dickson or Clifford algebra by an appropriate twist on the direct product $\Z{N}$. 

Given a twisted group algebra $V$ with element $x$, the transformation $L_x: V\mapsto V$ defined by $L_x(y)=xy$ is a linear transformation with standard matrix $[x]$. The product of elements $x,y\in V$ is $xy=[x]y.$ A result of Wedderburn \cite{W1925} guarantees that $\norm{xy}\le\sqrt{n}\,\norm{x}\norm{y}$ where $n$ is the order of $[x]$. If the twist is ``associative'' then $V$ is associative and $[xy]=[x][y]$ making $V$ isomorphic to a ring of square matrices. The Clifford algebras fall into this category, as well as the reals, complex numbers and quaternions. Cayley-Dickson algebras from the octonions upwards are non-associative.

 \section{Twisted Group Algebras}

 Let $V$ denote an $n$-dimensional inner product space over the ring $\mathbb{K}.$ Let $G$
 denote a group having only $n$ elements. Let $\mathcal{B}=\{i_p \mid p\in G \}$ denote a set of
 unit basis vectors for  $V.$ Then, for each $x\in V,$ there exist
 elements $\{x_p \mid x_p\in \mathbb{K}, p\in G \}$ such that
 $x=\Sum{p}x_pi_p.$\\
 Define a product on the elements of $\mathcal{B}$ and their negatives in the following manner.
 \\[12pt]
 Let $\sgn: G \times G \mapsto \{ -1,1\}$ denote a \emph{sign} function or `twist' on $G.$ Then for $p,q\in G$ define the product of $i_p$ and $i_q$ as follows.
 \begin{definition}
   \[i_p i_q = \sgn(p,q)i_{pq}\]
 \end{definition}
 Extend this product to $V$ in the natural way. That is,
 \begin{definition}
    \begin{align*}
    xy
    &= \left( \Sum{p}x_p i_p \right)\left(\Sum{q} y_q i_q \right)\\
    &= \Sum{p,q} x_p y_q i_p  i_q\\
    &= \Sum{p,q} \sgn(p,q) x_p y_q i_{pq}
    \end{align*}
 \end{definition}
 In defining the product this way, one gets the closure and distributive properties ``for free'', as well as $(cx)y=x(cy)=c(xy)$
 
 This product transforms the vector space $V$ into a \emph{twisted group algebra}. The properties of the algebra depend upon the properties of the twist and the properties of $G.$

 \begin{notation}
  Given a group $G$, twist $\alpha$ on $G$ and ring $\mathbb{K}$, let $\tga{G}{\alpha}{\mathbb{K}}$ denote the corresponding twisted group algebra. If $\mathbb{K}=\mathbb{R}$, abbreviate this notation $\tgar{G}{\alpha}$.
 \end{notation}
 \begin{definition}
 Let $x\in V=\tga{G}{\alpha}{\mathbb{K}}$. Define the matrix $\bracket{x}:G\times G\mapsto \mathbb{K}$ such that for $r,s\in G$, $\bracket{x}(r,s)=\sgn(rs^{-1},s)x_{rs^{-1}}.$
 \end{definition}
 \begin{theorem}\label{T:stMat}
 For $x,y\in V=\tga{G}{\alpha}{\mathbb{K}}$, $xy=\bracket{x}y.$
 \end{theorem}
 \begin{proof}
 Given $xy=\Sum{p,q} \sgn(p,q)x_p y_q  i_{pq}$ let $pq=r$. Then $p=rq^{-1},$ so 
 \[ xy=\Sum{r,q}\sgn(rq^{-1},q)x_{rq^{-1}}y_qi_r=\bracket{x}y. \]
 \end{proof}
 \begin{theorem}\label{T:iprs}
  $\bracket{i_p}(r,s)=\begin{cases}
                  \sgn(p,s) &\text{if\ }ps=r\\
		  0         &\text{otherwise}
                 \end{cases}$
 \end{theorem}
 \begin{corollary}\label{C:pqrs}
  $\bracket{\sgn(p,q)i_{pq}}(r,s)=\begin{cases}
                          \sgn(p,q)\sgn(pq,s) &\text{if\ }pqs=r\\
                          0          &\text{otherwise}
                         \end{cases}$
 
 \end{corollary}
 \begin{theorem}\label{T:ipqrs}
 $\left(\bracket{i_p}\bracket{i_q}\right)(r,s)=\begin{cases}
                        \sgn(p,qs)\sgn(q,s) &\text{if\ }pqs=r\\
                        0         &\text{otherwise}
                 \end{cases}$
 \end{theorem}
\begin{corollary}\label{C:assoc}
$\bracket{i_p}\bracket{i_q}=\sgn(p,q)\bracket{i_{pq}}$ provided $\sgn(p,q)\sgn(pq,s)=\sgn(p,qs)\sgn(q,s)$ for $p,q,s\in G$.
\end{corollary}
\section{Symmetric and anti-symmetric products}
\begin{definition}
 For group $G$, the \emph{interior} of $G\times G$, written $G_0^2$, is the set of all ordered pairs $(p,q)\in G\times G$ such that $p\ne e$, $q\ne e$ and $pq\ne e$ or equivalently, $e\ne p\ne q^{-1}\ne e.$
\end{definition}

\begin{definition}
 For group $G$ and twist $\sgn$, the \emph{symmetric} part of $G\times G$ with respect to $\sgn$, written $G^2+\sgn$, is the set of all ordered pairs in $G\times G$ such that $\sgn(p,q)=\sgn(q,p).$ The \emph{anti-symmetric} part of $G\times G$ with respect to $\sgn$, written $G^2-\sgn$, is the set of all ordered pairs in $G\times G$ such that $\sgn(p,q)\ne\sgn(q,p).$
\end{definition}
Every twisted product can be decomposed into its symmetric and anti-symmetric parts.
\begin{definition}
 For $x,y\in \tga{G}{\alpha}{\mathbb{K}}$, define the \emph{symmetric product} $x\vee y$ as
 \begin{equation*}
  x\vee y = \sum_{p,q}^{G^2+\sgn} \sgn(p,q)x_py_qi_{pq}
 \end{equation*}
and define the \emph{anti-symmetric product} $x\wedge y$ as
 \begin{equation*}
  x\wedge y = \sum_{p,q}^{G^2-\sgn} \sgn(p,q)x_py_qi_{pq}
 \end{equation*}
\end{definition}
The following results are immediate.
\begin{theorem}
 \begin{enumerate} For $x,y\in \tga{G}{\alpha}{\mathbb{K}}$,
  \item $xy=x\vee y + x\wedge y$
  \item $x\vee y = y\vee x$
  \item $x\wedge y = -(y\wedge x)$
  \item $x\vee y = \dfrac{xy+yx}{2}$
  \item $x\wedge y = \dfrac{xy-yx}{2}$
 \end{enumerate}

\end{theorem}

\begin{definition}
  For group $G$ and twist $\sgn$, the \emph{symmetric interior} of $G\times G$ with respect to $\sgn$, written $G_0^2+\sgn$, is the intersection of the symmetric part of $G\times G$ with the interior of $G\times G.$  The \emph{anti-symmetric interior} written $G_0^2-\sgn$, is the intersection of the anti-symmetric part of $G\times G$ with the interior of $G\times G.$
\end{definition}

\begin{theorem} \label{T:SAS}
If $\sgn(p,p^{-1})=\sgn(p^{-1},p)$ for $p\in G$ then
\begin{equation*}
\twisteq{x}{y}=x_ey + xy_e + \sum_p \sgn(p,p^{-1})x_py_{p^{-1}} -2x_ey_e + \sum_{p,q}^{G_0^2+\sgn}\sgn(p,q)\left(\frac{x_py_q+x_qy_p}{2}\right)i_{pq}
\end{equation*}
 and 
\begin{equation*}
\twistne{x}{y} = \sum_{p,q}^{G_0^2-\sgn} \sgn(p,q)\left(\frac{x_py_q-x_qy_p}{2}\right)i_{pq}
\end{equation*}
\end{theorem}

\begin{proof}
Separating out the terms of $\twisteq{x}{y}$ for which $p=e$, $q=e$, $p=q^{-1}$ and $e\ne p\ne q^{-1}\ne e$ yields

\begin{equation*}\label{E:twisteq} 
\twisteq{x}{y} =x_ey + xy_e + \sum_p \sgn(p,p^{-1})x_py_{p^{-1}} -2x_ey_e 
+ \sum_{p,q}^{G_0^2+\sgn}\sgn(p,q)\left(\frac{x_py_q+x_qy_p}{2}\right)i_{pq} 
\end{equation*}

The set of all $(p,q)\in G\times G$ for which $\sgn(p,q)\ne\sgn(q,p)$ is precisely the anti-symmetric interior of $G\times G$, since $\sgn(e,q)=\sgn(p,e)=1$ and $\sgn(p,p^{-1})=\sgn(p^{-1},p).$ Thus

\begin{equation*} \label{E:tne2} 
\twistne{x}{y} = \sum_{p,q}^{G_0^2-\sgn} \sgn(p,q)\left(\frac{x_py_q-x_qy_p}{2}\right)i_{pq} 
\end{equation*}
 \end{proof}

 \section{Twists and Ring Properties}

 Let $G$ denote a finite group with identity $e$, $\mathbb{K}$ a ring and $\sgn$ a twist on $G$. Let $V=\tga{G}{\sgn}{\mathbb{K}}.$

In order for $i_e$ to be the identity element 1 of the group algebra $V$, we require for all $p\in G$

\begin{equation}
 \sgn(e,p)=\sgn(p,e)=1
\end{equation}

 \begin{definition}
 If $\sgn(p,q)=\sgn(q,p)$ for all $p,q\in G$ , then $\sgn$ is a \emph{commutative} twist.
 \end{definition}
 \begin{theorem}
 If $V=\tga{G}{\sgn}{\mathbb{K}}$ and $\sgn$ is commutative, then the resulting product on $V$ is commutative.
 \end{theorem}
\begin{proof}
For $p,q\in G$, $i_pi_q=\sgn(p,q)i_{pq}=\sgn(q,p)i_{qp}$. Thus, for $x,y\in V$, $xy
= \Sum{p,q} x_p y_q \sgn(p,q) i_{pq} = \Sum{q,p} x_q y_p \sgn(q,p) i_{qp}=yx$
\end{proof}

 Corollary \ref{C:assoc} establishes that a twisted group algebra is isomorphic to to a ring of square matrices provided that a particular condition on the twist is satisfied, motivating the following definition.
 \begin{definition}\label{D:associative}
  If $\sgn(p,q)\sgn(pq,r)=\sgn(p,qr)\sgn(q,r)$ for $p,q,r\in G,$ then $\sgn$ is an \emph{associative} twist on $G.$
 \end{definition}
 \begin{theorem}\
  If $p,q,r\in G$ and $\sgn$ is associative, then $i_p \left(i_qi_r\right)=\left(i_p i_q\right)i_r.$
  \end{theorem}
\begin{proof}
  \begin{align*}
 i_p\left(i_qi_r\right)&=i_p\left(\sgn(q,r)i_{qr}\right)\\
                 &=\sgn(q,r)i_pi_{qr}=\sgn(p,qr)\sgn(q,r)i_{p(qr)}\\
                 &=\sgn(p,q)\sgn(pq,r)i_{(pq)r}\\
                 &=\sgn(p,q)i_{pq}i_r\\
                 &=\left(i_pi_q\right)i_r.
  \end{align*}
\end{proof}
\begin{theorem}\label{T:associative}
 If $x,y,z\in V=\tga{G}{\sgn}{\mathbb{K}}$ and if $\sgn$ is associative, then
 $x\left(yz\right)=\left(xy\right)z.$
\end{theorem}
\begin{proof}
 $y z=\Sum{q,r}\left(y_qz_r\right)i_qi_r,$ so
  \begin{align*}
 x(yz)&=
 \left(\Sum{p}x_pi_p\right)\left(\Sum{q,r}\left(y_qz_r\right)i_qi_r\right)\\
 &=\Sum{p,q,r}x_p\left(y_qz_r\right)i_p\left(i_qi_r\right)\\
 &=\Sum{p,q,r}\left(x_py_q\right)z_r\left(i_pi_q\right)i_r\\
 &=\left(\Sum{p,q}\left(x_py_q\right)i_pi_q\right)\Sum{r}z_ri_r\\
 &=\left(xy\right)z
  \end{align*}.
\end{proof}
 
 \begin{corollary}
 If $\sgn$ is associative, then $<\tga{G}{\sgn}{\mathbb{K}},+,\cdot>$ is a ring with unity.
 \end{corollary}

\begin{theorem}\label{T:lr_inverse}
  For each $p\in G,$ $\sgn\left(p,p^{-1}\right)i_{p^{-1}}$ and $\sgn\left(p^{-1},p\right)i_{p^{-1}}$
  are right and left inverses, respectively, of $i_p.$
 \end{theorem}
 \begin{proof}
  \begin{align*}
  i_p\left(\sgn\left(p,p^{-1}\right)i_{p^{-1}}\right)
    &=\sgn\left(p,p^{-1}\right)i_pi_{p^{-1}}\\
    &=\sgn\left(p,p^{-1}\right)\left(\sgn\left(p,p^{-1}\right)i_{pp^{-1}}\right)\\
    &=i_e=1\\
  \left(\sgn\left(p^{-1},p\right)i_{p^{-1}}\right)i_p
    &=\sgn\left(p^{-1},p\right)i_{p^{-1}}i_p\\
    &=\sgn\left(p^{-1},p\right)\left(\sgn\left(p^{-1},p\right)i_{pp^{-1}}\right)\\
    &=i_e=1
  \end{align*}
 \end{proof}
 \begin{definition}\label{D:invertive}
  If $\sgn\left(p,p^{-1}\right)=\sgn\left(p^{-1},p\right)$ for $p\in G,$ then $\sgn$ is an \emph{invertive} twist on $G.$
   \end{definition}
 \begin{theorem}
  If $p\in G$ and if $\sgn$ is invertive, then $i_p$ has an inverse \\
  $i_p^{-1}=\sgn\left(p,p^{-1}\right)i_{p^{-1}}=\sgn\left(p^{-1},p\right)i_{p^{-1}}.$
 \end{theorem}
 \begin{proof}
  Follows immediately from Theorem \ref{T:lr_inverse} and Definition \ref{D:invertive}.
 \end{proof}
 \begin{definition}\label{D:conjugate}
  If $\sgn$ is invertive, and $x\in V=\tga{G}{\sgn}{\mathbb{K}},$ then let $\conj{x}=\Sum{p}\conj{x}_pi_p^{-1}$ denote
  the \emph{conjugate} of $x.$
 \end{definition}
 \begin{theorem}\label{T:conjugate}
  If $\sgn$ is an invertive twist on $G,$ and if $x,y\in V=\tga{G}{\sgn}{\mathbb{K}},$ then
  \begin{enumerate}
   \item[(i)] $\conj{x}=\Sum{p}\sgn\left(p^{-1},p\right)\conj{x}_{p^{-1}}i_p$
   \item[(ii)] $\conj{\conj{x}}=x$
   \item[(iii)] $\conj{(x+y)}=\conj{x}+\conj{y}$
   \item[(iv)] $\conj{(cx)}=\conj{c}\conj{x}$ for all $c\in \mathbb{K}.$
  \end{enumerate}
 \end{theorem}
  \begin{proof}
   \begin{enumerate}
    \item[]
    \item[(i)] Let $q^{-1}=p.$ Then 
     \begin{align*}
       \conj{x}&=\Sum{q}\conj{x}_qi_q^{-1}\\
                &=\Sum{q}\conj{x}_q\sgn\left(q,q^{-1}\right)i_{q^{-1}}\\
                &=\Sum{p}\sgn\left(p^{-1},p\right)\conj{x}_{p^{-1}}i_p
     \end{align*}
    \item[(ii)] Let $z=\conj{x}.$ Then
     \begin{align*} z&=\Sum{p}\sgn\left(p^{-1},p\right)\conj{x}_{p^{-1}}i_p\\
                           &=\Sum{p}z_pi_p
       \end{align*}
     where $z_p=\sgn\left(p^{-1},p\right)\conj{x}_{p^{-1}}.$ Then $z_{p^{-1}}=\sgn\left(p,p^{-1}\right)\conj{x}_p,$ and
     \begin{align*}
     \conj{z}_{p^{-1}}&=\sgn\left(p,p^{-1}\right)x_p\text{.\ So}\\ 
     \conj{\conj{x}}&=\conj{z}\\
              &=\Sum{p}\sgn\left(p^{-1},p\right)\conj{z}_{p^{-1}}i_p\\
              &=\Sum{p}\sgn\left(p^{-1},p\right)\sgn\left(p,p^{-1}\right)x_pi_p \\
              &=\Sum{p}x_pi_p=x
       \end{align*}
    \item[(iii)] 
     \begin{align*}
        \conj{\left(x+y\right)}&=\Sum{p}\conj{\left(x_p+y_p\right)}i_p^{-1}\\
                                    &=\Sum{p}\left(\conj{x}_p+\conj{y}_p\right)i_p^{-1}\\
                                    &=\Sum{p}\conj{x}_pi_p^{-1} +\Sum{p}\conj{y}_pi_p^{-1}\\
                                    &=\conj{x}+\conj{y}.
       \end{align*}
    \item[(iv)] 
     \begin{align*}
        \conj{(cx)}&=\Sum{p}\conj{(cx_p)}i_p^{-1}\\
                  &=\Sum{p}\conj{c}\conj{x}_pi_p^{-1}\\
                  &=\conj{c}\Sum{p}\conj{x}_pi_p^{-1}\\
                  &=\conj{c}\conj{x}
       \end{align*}
   \end{enumerate}
  \end{proof}
  The invertive property is not sufficient for establishing the algebraic property 
$\conj{(xy)}=\conj{y}\conj{x}$. For that, we need the \emph{propreity} property.
\section{Proper Twists}
  \begin{definition}\label{D:proper}
   The statement that the twist $\sgn$ on $G$ is  \emph{proper} means that if $p,q\in G,$ then
   \begin{enumerate}
    \item[(1)] $\sgn(p,q)\sgn\left(q,q^{-1}\right)=\sgn\left(pq,q^{-1}\right)$
    \item[(2)] $\sgn\left(p^{-1},p\right)\sgn(p,q)=\sgn\left(p^{-1},pq\right).$
   \end{enumerate}
  \end{definition}
  \begin{theorem}\label{T:assocprod}
   Every associative twist is proper.
  \end{theorem}
  \begin{proof}\qquad\\
    \begin{enumerate}
     \item[(1)] $\sgn(p,q)\sgn\left(q,q^{-1}\right)
                 =\sgn\left(pq,q^{-1}\right)\sgn\left(p,qq^{-1}\right)\\
                 =\sgn\left(pq,q^{-1}\right)\sgn(p,e)=\sgn\left(pq,q^{-1}\right)\sgn(e,e)
                 =\sgn\left(pq,q^{-1}\right)$
     \item[(2)] $\sgn\left(p^{-1},p\right)\sgn(p,q)
                 =\sgn\left(p^{-1}p,q\right)\sgn\left(p^{-1},pq\right)\\
                 =\sgn(e,q)\sgn\left(p^{-1},pq\right)=\sgn(e,e)\sgn\left(p^{-1},pq\right)
                 =sgn\left(p^{-1},pq\right)$
     \end{enumerate}
  \end{proof}
  \begin{theorem}\label{T:conjid}
   Every proper twist is invertive.
  \end{theorem}
  \begin{proof}
    Suppose $\sgn$ is proper.
    Then $\sgn\left(p,p^{-1}\right)\sgn\left(p^{-1},p\right)=\sgn\left(pp^{-1},p\right)=\sgn(e,p)=1,$ thus $\sgn\left(p,p^{-1}\right)=\sgn\left(p^{-1},p\right).$ So $\sgn$ is invertive.
  \end{proof}

  \begin{theorem}
  If $\sgn$ is a proper twist on $G,$ then
   $\conj{(i_pi_q)}=\conj{i}_q\conj{i}_p$ for all $p,q\in G.$
  \end{theorem}
  \begin{proof}
 Since $i_pi_q=\sgn(p,q)i_{pq},$ $\conj{\left(i_pi_q\right)}
   =\conj{\left(\sgn(p,q)i_{pq}\right)}=\sgn(p,q)\conj{\left(i_{pq}\right)}\\
   =\sgn(p,q)i_{pq}^{-1} =\sgn(p,q)\sgn\left(\left(pq\right)^{-1},pq\right)i_{(pq)^{-1}}.$\\
   On the other hand,  
   \begin{align*}
   \conj{i}_q\conj{i}_p=i_q^{-1}i_p^{-1}&=
   \left(\sgn\left(q^{-1},q\right)i_{q^{-1}}\right)\left(\sgn\left(p^{-1},p\right)i_{p^{-1}}\right)\\
   &=\sgn\left(q^{-1},q\right)\sgn\left(p^{-1},p\right)\sgn\left(q^{-1},p^{-1}\right)i_{q^{-1}p^{-1}}\\
   &=\sgn\left(q^{-1},q\right)\sgn\left(p^{-1},p\right)\sgn\left(q^{-1},p^{-1}\right)i_{(pq)^{-1}}
   \end{align*}
   Therefore, in order to show that $\conj{(i_pi_q)}
   =\conj{i}_q\conj{i}_p,$ it is sufficient to show that\\
   $\sgn(p,q)\sgn\left(\left(pq\right)^{-1},pq\right)=\sgn\left(q^{-1},q\right)\sgn\left(p^{-1},p\right)\sgn\left(q^{-1},p^{-1}\right).$\\
   Beginning with the expression on the left,
   \begin{align*}
   \sgn(p,q)\sgn\left((pq)^{-1},pq\right) 
&=\sgn(p,q)\sgn\left(q,q^{-1}\right)\sgn\left(q,q^{-1}\right)\sgn\left((pq)^{-1},pq\right)\\
&=\sgn\left(pq,q^{-1}\right)\sgn\left(q,q^{-1}\right)\sgn\left((pq)^{-1},pq\right)\\
&=\sgn\left((pq)^{-1},pq\right)\sgn\left(pq,q^{-1}\right)\sgn\left(q,q^{-1}\right)\\
&=\sgn\left((pq)^{-1},p\right)\sgn\left(q,q^{-1}\right)\\
&=\sgn\left((pq)^{-1},p\right)\sgn\left(p,p^{-1}\right)\sgn\left(p,p^{-1}\right)\sgn\left(q,q^{-1}\right)\\
&=\sgn\left(q^{-1}p^{-1},p\right)\sgn\left(p,p^{-1}\right)\sgn\left(p,p^{-1}\right)\sgn\left(q,q^{-1}\right)\\
&=\sgn\left(q^{-1},p^{-1}\right)\sgn\left(p,p^{-1}\right)\sgn\left(q,q^{-1}\right)\\
&=\sgn\left(q^{-1},p^{-1}\right)\sgn\left(p^{-1},p\right)\sgn\left(q^{-1},q\right)
   \end{align*}
  \end{proof}
  \begin{theorem}
  If $\sgn$ is a proper twist on $G$ and if $x,y\in V=\tga{G}{\sgn}{\mathbb{K}},$ then
  $\conj{(xy)}=\conj{y}\conj{x}.$
  \end{theorem}
  \begin{proof}
   \begin{align*}
     \conj{(xy)}&=\conj{\left(\left(\Sum{p}x_pi_p\right)\left(\Sum{q}y_qi_q\right)\right)}\\
                     &=\conj{\left(\Sum{p,q}x_py_qi_pi_q\right)}\\
                     &=\Sum{p,q}\conj{(x_py_qi_pi_q)}\\
                     &=\Sum{q,p}\conj{y}_q\,\conj{x}_pi_q^{-1}i_p^{-1}\\
                     &=\left(\Sum{q}\conj{y}_qi_q^{-1}\right)\left(\Sum{p}\conj{x}_pi_p^{-1}\right)\\
                     &=\conj{y}\conj{x}
   \end{align*}
  \end{proof}
  \begin{definition}
   The \emph{inner product} of elements $x$ and  $y$ in $V=\tga{G}{\sgn}{\mathbb{K}}$ is $\ip{x}{y}
   =\Sum{p}x_p\conj{y}_p.$
  \end{definition}
  \begin{theorem}\label{T:product}
   If $\sgn$ is a proper twist on $G,$ and if $x,y\in V=\tga{G}{\sgn}{\mathbb{K}},$ then the product $xy$ has the Fourier expansion
   \[ xy=\Sum{r}\ip{x}{i_r\conj{y}}i_r =\Sum{r}\ip{y}{\conj{x}i_r}i_r
   \]
  \end{theorem}
  \begin{proof}
  ~\\
  (i) $\conj{y}
  =\Sum{s}\sgn\left(s,s^{-1}\right)\conj{y}_{s^{-1}}i_s.$ 
  Let $p=rs.$ Then
  \begin{align*} 
  i_r\conj{y}& =\Sum{s}\sgn\left(s,s^{-1}\right)\conj{y}_{s^{-1}}i_r i_s\\
  &=\Sum{s}\sgn(r,s)\sgn\left(s,s^{-1}\right)\conj{y}_{s^{-1}}i_{rs}\\
  &=\Sum{s}\sgn\left(rs,s^{-1}\right)\conj{y}_{s^{-1}}i_{rs}\\
  &=\Sum{p}\sgn\left(p,p^{-1}r\right)\conj{y}_{p^{-1}r}i_p
  \end{align*}
 Thus,
\begin{align*}
   \ip{x}{i_r \conj{y}}
    &=\Sum{p}\left(p,p^{-1}r\right)x_p\conj{\conj{y}}_{p^{-1}r}\\
    &=\Sum{p}\sgn\left(p,p^{-1}r\right)x_py_{p^{-1}r}
  \end{align*}
  (ii) $\conj{x}=\Sum{s}\sgn\left(s^{-1},s\right)\conj{x}_{s^{-1}}i_s.$ Let $q=sr.$ Then
   \begin{align*}
    \conj{x}i_r& =\Sum{s}\sgn\left(s^{-1},s\right)\conj{x}_{s^{-1}}i_si_r\\
                      & =\Sum{s}\sgn\left(s^{-1},s\right)\sgn(s,r)\conj{x}_{s^{-1}}i_{sr}\\
                      & =\Sum{s}\sgn\left(s^{-1},sr\right)\conj{x}_{s^{-1}}i_{sr}\\
                      & = \Sum{q}\sgn\left(rq^{-1},q\right)\conj{x}_{rq^{-1}}i_q
   \end{align*} 
   So $\ip{ y}{\conj{x}i_r } = \Sum{q}\sgn\left(rq^{-1},q\right)x_{rq^{-1}}y_q.$\\
  (iii) 
   \begin{align*}
     xy &= \Sum{p,q}x_py_qi_pi_q\\
                  &= \Sum{p,q}\sgn(p,q)x_py_qi_{pq}\\
                  &= \Sum{q,p}\sgn(p,q)x_py_qi_{pq}
   \end{align*} 
Let $pq=r.$ Then $p=rq^{-1}$ and $q=p^{-1}r.$\\
   Thus,\\
   $xy=\Sum{r,q}\sgn\left(rq^{-1},q\right)x_{rq^{-1}}y_qi_r=\Sum{r}\ip{y}{\conj{x}i_r}i_r.$\\
   And\\
   $xy=\Sum{r,p}\sgn\left(p,p^{-1}r\right)x_py_{p^{-1}r}i_r=\Sum{r}\ip{x}{i_r\conj{y}}i_r.$
  \end{proof}
\begin{corollary}
 If $\sgn$ is proper, then $\ip{xy}{i_r}=\ip{x}{i_r\conj{y}}=\ip{y}{\conj{x}i_r}.$
\end{corollary}
\begin{definition}
A \emph{proper} *-algebra is an inner product space with an involution $(*)$ satisfying the \emph{adjoint} properties:
\begin{eqnarray}
\ip{xy}{z}=\ip{y}{\conj{x}z}\\
\ip{x}{yz}=\ip{x\conj{z}}{y}
\end{eqnarray} for all $x,y,z$ in the algebra.
\end{definition}
Note: The Cayley-Dickson algebras are known to be proper *-algebras \cite{M1998}.
\begin{theorem}
 If $V=\tga{G}{\sgn}{\mathbb{K}}$ and if $\sgn$ is proper, then $V$ is a proper *-algebra.
\end{theorem}
\begin{proof}
 Let $x,y,z\in G$. Then 
\begin{eqnarray*}
 & \ip{xy}{z}&=\ip{xy}{\sum_rz_ri_r}\\
 &           &=\sum_r\ip{xy}{z_ri_r}\\
 &           &=\sum_r\conj{z}_r\ip{xy}{i_r}\\
 &           &=\sum_r\conj{z}_r\ip{y}{\conj{x}i_r}\\
 &           &=\sum_r\ip{y}{\conj{x}\left(z_ri_r\right)}\\
 &           &=\ip{y}{\conj{x}\sum_rz_ri_r}\\
 &           &=\ip{y}{\conj{x}z}
\end{eqnarray*}
and
\begin{eqnarray*}
 & \ip{x}{yz}&=\ip{\sum_rx_ri_r}{yz}\\
 &           &=\sum_r\ip{x_ri_r}{yz}\\
 &           &=\sum_rx_r\ip{i_r}{yz}\\
 &           &=\sum_rx_r\conj{\ip{yz}{i_r}}\\
 &           &=\sum_rx_r\conj{\ip{y}{i_r\conj{z}}}\\
 &           &=\sum_rx_r\ip{i_r\conj{z}}{y}\\
 &           &=\sum_r\ip{\left(x_ri_r\right)\conj{z}}{y}\\
 &           &=\ip{\sum_r\left(x_ri_r\right)\conj{z}}{y}\\
 &           &=\ip{x\conj{z}}{y}
\end{eqnarray*}
\end{proof}
The next theorem and its corollaries apply specifically to real Cayley-Dickson and Clifford algebras.
  \begin{theorem}\label{T:zero} Suppose $V=\tgar{G}{\sgn}$, $\sgn$ is proper and $p=p^{-1}$ for all $p\in G.$ Then for all $p\in G,$
   $\left[1-\sgn(p,p)\right]\ip{x}{i_px}=0.$
  \end{theorem}
  \begin{proof}
   $i_px=\sum_qx_qi_pi_q=\sum_q\sgn(p,q)x_qi_{pq}=\sum_r\sgn(p,pr)x_{pr}i_r$ where $q=pr.$
   
   \begin{align*}
   2\ip{x}{i_px}&=2\sum_r\sgn(p,pr)x_rx_{pr}\\
                &=2\,\sgn(p,p)\sum_r\sgn(p,r)x_rx_{pr}\\
		&=\sgn(p,p)\sum_r\sgn(p,r)x_rx_{pr}+\sgn(p,p)\sum_q\sgn(p,q)x_qx_{pq}\\
		&=\sgn(p,p)\sum_r\sgn(p,r)x_rx_{pr}+\sgn(p,p)\sum_r\sgn(p,pr)x_{pr}x_r\\
		&=\sgn(p,p)\sum_r\sgn(p,r)x_rx_{pr}+\sgn(p,p)\sgn(p,p)\sum_r\sgn(p,r)x_rx_{pr}\\
		&=\left[\sgn(p,p)+1\right]\sum_r\sgn(p,r)x_rx_{pr}\\
		&=\left[\sgn(p,p)+1\right]\sgn(p,p)\sum_r\sgn(p,pr)x_rx_{pr}\\
		&=\left[1+\sgn(p,p)\right]\ip{x}{i_px}
   \end{align*}
   Thus $\left[1-\sgn(p,p)\right]\ip{x}{i_px}=0.$
  \end{proof}
  \begin{corollary}
   Suppose $V=\tgar{G}{\sgn}$, $\sgn$ is proper and $p=p^{-1}$ for all $p\in G.$ Then $\ip{x}{i_px}=0$ provided $\sgn(p,p)=-1.$
  \end{corollary}
  \begin{corollary}
   Suppose $V=\tgar{G}{\sgn}$, $\sgn$ is proper and $p=p^{-1}$ for all $p\in G.$ Then $\ip{x\conj{x}}{i_p}=0$ provided $\sgn(p,p)=-1.$
  \end{corollary}

\section{The Cayley-Dickson Algebras}

The complex numbers can be constructed as ordered pairs of real numbers, and the quaternions as ordered pairs of complex numbers. The Cayley-Dickson process continues this development. Ordered pairs of quaternions are octonions and ordered pairs of octonions are sedenions. For each non-negative integer $N$ the Cayley-Dickson algebra $\cds{N}$ is a properly twisted group algebra on $\mathbb{R}^{2^N}$, with $\cds{0}=\mathbb{R}$ denoting the reals, $\cds{1}=\mathbb{C}$ the complex numbers, $\cds{2}=\mathbb{H}$ the quaternions and $\cds{3}=\mathbb{O}$ the octonions \cite{B2001}.
 
 In the remainder of the paper we will develop the real Cayley-Dickson and Clifford algebras as proper subsets of the Hilbert space $\ell^2$ of square summable sequences. The canonical basis for $\ell_2$ will be indexed by the group $\mathbb{Z}^+=\{0,1,2,3,\cdots\}$ of non-negative integers with group operation the `bit-wise exclusive or' of the binary representations of the elements.

Equate a real number $r$ with the sequence $r,0,0,0,\cdots$. Given two real number sequences $x=x_0,x_1,x_2,\cdots$ and $y=y_0,y_1,y_2,\cdots$, equate the ordered pair $\left( x,y \right)$ with the `shuffled' sequence
 \[\left( x,y \right)=x_0,y_0,x_1,y_1,x_2,y_2,\cdots\]
 
 The sequences
 \begin{align*}
 \noindent i_0 \quad &= 1,0,0,0,\cdots\\
 i_1 \quad &= 0,1,0,0,\cdots\\
 i_2 \quad &= 0,0,1,0,\cdots\\
 i_3 \quad &= 0,0,0,1,\cdots\\
 &\quad\vdots \\
 i_{n-1}&
\end{align*}

 form the canonical basis for $\mathbb{R}^n$ and satisfy the identities
 \begin{eqnarray}
 \noindent i_{2n}  &=\left( i_n,0 \right)\\
           i_{2n+1}&=\left( 0,i_n\right)
 \end{eqnarray}

 This produces a numbering of the unit basis vectors for the Cayley-Dickson algebras which differs from other numberings, yet it arises naturally from equating ordered pairs with shuffled sequences. It also produces, quite naturally, the Cayley-Dickson conjugate identity.
 
 In order to define a conjugate satisfying $x+\conj{x}\in\mathbb{R}$, we must have $\conj{i}_0=i_0$ (else 1 will not be the identity) and $\conj{i}_n=-i_n$ for $n>0.$ This leads to the result
	 
\begin{equation}\label{E:conj}
 \conj{(x,y)}=\left( \conj{x},-y\right) 
\end{equation}

The usual way of multiplying ordered pairs of real numbers $(a,b)$ and $(c,d)$ regarded as complex numbers is

\[ (a,b)(c,d)=(ac-bd,ad+bc) \]

This method of multiplying ordered pairs, if repeated for ordered pairs of ordered pairs ad infinitum, produces a sequences of algebras of dimension $2^N$. However, the four dimensional algebra produced when $N=2$ is not the quaternions but a four dimensional algebra with zero divisors. The twist on $\mathbb{S}_n=\mathbb{R}^{2^n}$ produced by this product is $\hdm(p,q)=(-1)^{\sob{p\wedge q}}$ with $p\wedge q$ the bitwise `and' of the binary representations of $p$ and $q$ and $\sob{p}$ the `sum of the bits' of the binary representation of $p.$  $i_pi_q=(-1)^{<p\wedge q>}i_{pq}$ where $pq=p\veebar q$: the `bit-wise' ``exclusive or'' of $p$ and $q.$ (This group product is equivalent to addition in $\mathbb{Z}_2^N.$) Since $\hdm$, considered as a matrix, is a Hadamaard matrix, $\hdm(p,q)=(-1)^{\sob{p\wedge q}}$ may be termed the `Hadamaard twist'. It is a simple exercise to show that the Hadamaard twist is associative.

The product which produces the quaternions from the complex numbers is the Cayley-Dickson product

\begin{equation}\label{E:pairproduct}
(a,b)(c,d)=(ac-d\conj{b},\conj{a}d+cb) 
\end{equation}

For real numbers $a,b,c,d$ this is the complex number product. For complex numbers $a,b,c,d$ this is the quaternion product. For quaternions, it is the octonion product, etc.

One may establish immediately that $i_0=(1,0)$ is both the left and the right identity.
Furthermore, applying this product to all the unit basis vectors yields the following identities:

\begin{eqnarray}
&i_{2p}i_{2q}    &=\left( i_p,0 \right)\left( i_q,0 \right)= \left(i_pi_q,0 \right)\\
&i_{2p}i_{2q+1}  &=\left( i_p,0 \right)\left( 0,i_q \right)= \left( 0,\conj{i}_pi_q\right)\\
&i_{2p+1}i_{2q}  &=\left( 0,i_p \right)\left( i_q,0 \right)= \left( 0,i_qi_p \right)\\
&i_{2p+1}i_{2q+1}&=\left( 0,i_p \right)\left( 0,i_q \right)=-\left( i_q\conj{i}_p,0 \right)
\end{eqnarray}

This product on the unit basis vectors recursively defines a product and a twist $\cyd$ on the indexing sets $\Z{N}$ for each Cayley-Dickson space $\cds{N}$ in such a way that $\cds{N}\subset\cds{N+1}\subset\ell^2$ for all $N$. The product $pq$ of elements of $\Z{N}$ implied by these identities is the bit-wise `exclusive or' $p\veebar q$ of the binary representations of $p$ and $q,$ which is equivalent to addition in $\Z{N}.$ For this product, 0 is the identity and $p^{-1}=p$ for all $p\in \Z{N}.$ 

These identities imply the following defining properties of the Cayley-Dickson twist $\cyd:$

 \begin{eqnarray}\label{P:cyd1}
   \cyd(0,0)&=&\cyd(p,0)=\cyd(0,q)=1\\     
   \cyd(2p,2q)&=&\cyd(p,q)\\     
   \cyd(2p+1,2q)&=&\cyd(q,p)\\    
   \cyd(2p,2q+1)&=&\begin{cases} -\cyd(p,q)&\text{if\ }p\ne0\\                                                       1 &\text{otherwise}\end{cases}
 \end{eqnarray}
 \begin{eqnarray}\label{P:cyd5}
  \cyd(2p+1,2q+1)&=&\begin{cases}\cyd(q,p)&\text{if\ }p\ne0\\
                                                       -1&\text{otherwise}    
                                          \end{cases}
 \end{eqnarray}

These properties, in turn, imply the \emph{quaternion} properties:

For $0\ne p\ne q\ne 0$,
 \begin{eqnarray}
  \cyd(p,p)&=&-1\\
  \cyd(p,q)&=&-\cyd(q,p)\\
  \cyd(p,q)&=&\cyd(q,pq)=\cyd(pq,p)
 \end{eqnarray}

The quaternion properties imply that if $0\ne p\ne q\ne 0$ and if $i_pi_q=i_r$ then $i_pi_q=i_qi_r=i_ri_p$ and $i_pi_q=-i_qi_p.$ Thus elements of $\Z{k}$ for $k\ge2$ can be arranged into triplets of the form $(p,q,pq)$ for which $i_pi_q=i_{pq}.$ For $\Z{3}$ these are $(1,2,3),(1,4,5),(1,6,7),(2,4,6),(2,7,5),(3,6,5)$ and $(3,7,4).$

The quaternion properties also imply that $\cyd$ is a proper sign function. That is, for all $p,q$,
 \begin{eqnarray}
  \cyd(p,q)\cyd(q,q)=\cyd(pq,q)\\
  \cyd(p,p)\cyd(p,q)=\cyd(p,pq)
 \end{eqnarray}

Thus Cayley-Dickson algebras are proper *-algebras satisfying the adjoint properties:
For all $x,y,z,$
\begin{eqnarray*}
 \ip{xy}{z}=\ip{y}{\conj{x}z}\\
 \ip{x}{yz}=\ip{x\conj{z}}{y}
\end{eqnarray*}

Left open is the question whether $\ell^2$ is the completion of the sequence of Cayley-Dickson algebras. It is clear that if $x\in\ell^2$ then $xy\in\ell^2$ provided $y\in\cds{N}.$ It is not clear, however that $xy\in\ell^2$ for all $x,y\in\ell^2.$

Since the interior of $\Z{N}\times\Z{N}$ is anti-symmetric, by theorem \ref{T:SAS} we have

\begin{theorem}
\begin{eqnarray*}
&\twisteq{x}{y}&=x_0y + y_0x + \ip{x}{\conj{y}} -2x_0y_0\\
&              &=x_0y + y_0x -\ip{x}{y}
\end{eqnarray*}

\end{theorem}

\section{The Cayley-Dickson Twist on $\mathbb{Z}^+$}

The recursive definition \ref{P:cyd1}--\ref{P:cyd5} of the Cayley-Dickson twist on $\mathbb{Z}^+$ may be restated as follows:
For $p,q\in\mathbb{Z}^+$ and $r,s\in\{0,1\}$,
\begin{eqnarray}
 \cyd(0,0)      &=&1\\
 \cyd(2p+r,2q+s)&=&\cyd(p,q)A_{pq}(r,s)
\end{eqnarray}
where 
\begin{eqnarray}
A_{pq} 
 & = & \left(\begin{array}{rr}
                      1 & 1\\
		      1 &-1
                     \end{array}\right) \text{\ if\ }p=0\\
 & = & \left(\begin{array}{rr}
                      1 & -1\\
		      1 &  1
                     \end{array}\right) \text{\ if\ }0\ne p=q \text{\ or\ }p\ne q=0\\
 & = & \left(\begin{array}{rr}
                       1 & -1\\
		      -1 & -1
                     \end{array}\right)  \text{\ if\ }0\ne p\ne q\ne0.
\end{eqnarray}

Define $\cyd_0=(1)$. Then, for each non-negative integer $N$, $\cyd_{N+1}$ is a partitioned matrix defined by
\begin{equation}
 \cyd_{N+1}=\left(\cyd_N(p,q)A_{pq}\right)
\end{equation}

The twist for the sedenions, octonions, quaternions, complex numbers and reals is given by the matrix $\cyd_4:$

$\left(
 \begin{array}{rr|rr|rr|rr|rr|rr|rr|rr}
 1 &  1 &  1 &  1 &  1 &  1 &  1 &  1 &  1 &  1 &  1 &  1 &  1 &  1 &  1 &  1\\
 1 & -1 &  1 & -1 &  1 & -1 &  1 & -1 &  1 & -1 &  1 & -1 &  1 & -1 &  1 & -1\\
\hline 
 1 & -1 & -1 &  1 &  1 & -1 & -1 &  1 &  1 & -1 & -1 &  1 &  1 & -1 & -1 &  1\\
 1 &  1 & -1 & -1 & -1 & -1 &  1 &  1 & -1 & -1 &  1 &  1 & -1 & -1 &  1 &  1\\
\hline 
 1 & -1 & -1 &  1 & -1 &  1 &  1 & -1 &  1 & -1 & -1 &  1 & -1 &  1 &  1 & -1\\
 1 &  1 &  1 &  1 & -1 & -1 & -1 & -1 & -1 & -1 &  1 &  1 &  1 &  1 & -1 & -1\\
\hline 
 1 & -1 &  1 & -1 & -1 &  1 & -1 &  1 & -1 &  1 & -1 &  1 &  1 & -1 &  1 & -1\\
 1 &  1 & -1 & -1 &  1 &  1 & -1 & -1 &  1 &  1 &  1 &  1 & -1 & -1 & -1 & -1\\
\hline 
 1 & -1 & -1 &  1 & -1 &  1 &  1 & -1 & -1 &  1 &  1 & -1 &  1 & -1 & -1 &  1\\
 1 &  1 &  1 &  1 &  1 &  1 & -1 & -1 & -1 & -1 & -1 & -1 & -1 & -1 &  1 &  1\\
\hline 
 1 & -1 &  1 & -1 &  1 & -1 &  1 & -1 & -1 &  1 & -1 &  1 & -1 &  1 & -1 &  1\\
 1 &  1 & -1 & -1 & -1 & -1 & -1 & -1 &  1 &  1 & -1 & -1 &  1 &  1 &  1 &  1\\
\hline 
 1 & -1 & -1 &  1 &  1 & -1 & -1 &  1 & -1 &  1 &  1 & -1 & -1 &  1 &  1 & -1\\
 1 &  1 &  1 &  1 & -1 & -1 &  1 &  1 &  1 &  1 & -1 & -1 & -1 & -1 & -1 & -1\\
\hline 
 1 & -1 &  1 & -1 & -1 &  1 & -1 &  1 &  1 & -1 &  1 & -1 & -1 &  1 & -1 &  1\\
 1 &  1 & -1 & -1 &  1 &  1 &  1 &  1 & -1 & -1 & -1 & -1 &  1 &  1 & -1 & -1
\end{array}
\right)$\\

\section{Clifford Algebra}
In Clifford algebra \cite{BS1970}, the same unit basis vectors $\mathcal{B}=\{i_p | p\in \Z{N}\}$ and the same bit-wise `exclusive or' group operation on $\Z{N}$ may be used. Only the twist $\clf$  will differ.

In Clifford algebra, the unit basis vectors are called `blades'. Each blade has a numerical `grade'. The grade of the blade $i_p$ is the sum of the bits of the binary representation of $p$ denoted here by $\sob{p}.$

$i_0=1$ is the unit scalar, and is a 0-blade.

$i_1,i_2,i_4,\cdots,i_{2^n}\cdot$ are 1-blades, or `vectors' in Clifford algebra parlance.

$i_3,i_5,i_6,\cdots$ are 2-blades or `bi-vectors'. 

$i_7,i_{11},i_{13},i_{14},\cdots$ are 3-blades or `tri-vectors', etc.

As was the case with Cayley-Dickson algebras, this is not the standard notation for the basis vectors. However, it has the advantage that the product of unit basis vectors satisfies $i_pi_q=\clf(p,q)i_{pq}$ for a suitably defined Clifford twist $\clf$ on $\Z{N}.$

In the standard notation, 1-blades or `vectors' are denoted $e_1,e_2,e_3,\cdots,$ whereas 2-blades or `bivectors' are denoted $e_{12}, e_{13}, e_{23}, \cdots$ etc.

Translating from the $e$-notation to the $i$-notation is straightforward. For example, the 3-blade $e_{134}$ translates as $i_{13}$ since the binary representation of 13 is 1101 with bits 1, 3 and 4 set. The 2-blade $e_{23}=i_6$ since the binary representation of 6 is 110, with bits 2 and 3 set.

All the properties of $n$-blades can be deduced from five fundamental properties:
\begin{enumerate}
\item The square of 1-blades is 1. 
\item The product of 1-blades is anticommutative.
\item The product of 1-blades is associative.
\item The conjugate of a 1-blade is its negative.
\item Every $n$-blade can be factored into the product of $n$ distinct 1-blades.
\end{enumerate}

The convention is that, if $j<k,$ then $e_je_k=e_{jk},$ thus $e_ke_j=-e_{jk}.$

Any two $n$-blades may be multiplied by first factoring them into 1-blades. For example, the product of $e_{134}$ and $e_{23},$ is computed as follows:
\begin{align*} e_{134}e_{23}&=e_1e_3e_4e_2e_3\\
&=-e_1e_4e_3e_2e_3\\
&=e_1e_4e_2e_3e_3\\
&=e_1e_4e_2\\
&=-e_1e_2e_4\\
&=-e_{124}
\end{align*}

Since $e_{134}=i_{13}$ and $e_{23}=i_6,$ and the bit-wise `exclusive or' of 13 and 6 is 11, the same product using the `$i$' notation is
\[ i_{13}i_6=\clf(13,6)i_{11} \] so evidently, $\clf(13,6)=-1.$

The Clifford twist $\clf$ can be defined recursively as follows:
\begin{eqnarray}
\clf(0,0)&=&1\\
\clf(2p+1,2q)&=&\clf(2p,2q)=\clf(p,q)\\
\clf(2p+1,2q+1)&=&\clf(2p,2q+1)=\sbf{p}\clf(p,q)
\end{eqnarray}

For Clifford algebra, $x+\conj{x}$ is \emph{not} generally a real number.

Since
\begin{equation}
 \conj{i}_p=\clf(p,p)i_p
\end{equation}
and since it can be shown that
\begin{equation} \clf(p,p)=(-1)^s \text{\ where\ } s=\frac{\sob{p}\left(\sob{p}+1\right)}{2}
\end{equation}
it follows that, if $i_p$ is an $n$-blade, then 
\begin{equation}
 \conj{i}_p=\begin{cases}
              -i_p &\text{if\ }n=4k+1\text{\ or\ }n=4k+2\\
               i_p &\text{if\ }n=4k\text{\ or\ }n=4k+3
            \end{cases}
\end{equation}

This divides the $n$-blades into the `real' blades and the `imaginary' blades, where
\begin{enumerate} 
\item $0$-blades are real
\item $1$-blades are imaginary
\item $(n+2)$-blades have the opposite `parity' of $n$-blades.
\end{enumerate}

 The Clifford twist $\clf$ can be shown to be associative: 
\begin{equation}\clf(p,q)\clf(pq,r)=\clf(p,qr)\clf(q,r)\end{equation}

Since associative twists are proper, Clifford algebras are proper *-algebras. That is, for multivectors $x,y$ and $z$,
\begin{equation} \ip{xy}{z}=\ip{y}{\conj{x}z}\end{equation} \begin{equation}\ip{x}{yz}=\ip{x\conj{z}}{y}\end{equation}

\section{Conclusion}

There are a number of interesting questions about \emph{properly} twisted group algebras.

The union $\displaystyle{ \cds{\infty}=\bigcup_{N=0}^\infty\cds{N}}$ of all the finite dimensional Cayley-Dickson algebras is itself a Cayley-Dickson algebra. However, it is incomplete. The Hilbert space $\ell^2\supset\cds{\infty}$ of square-summable sequences is a natural completion for $\cds{\infty}.$ Furthermore, for $x,y\in\ell^2$ the product $xy$ is a well-defined sequence. However, the product $xy$ is not obviously square-summable so it is not clear that $\ell^2$ is closed under the Cayley-Dickson product. It is easily shown, though, that if $x\in\ell^2$ and $y\in\cds{N}$, then $xy\in\ell^2.$ 

\begin{question}
 Is the Hilbert space $\ell^2$ of square-summable sequences a Cayley-Dickson algebra?
\end{question}

It follows from a result of Wedderburn \cite{W1925} that for $x,y\in\cds{N}$ $\norm{xy}\le\sqrt{2^N}\norm{x}\norm{y}$. Computer generation of random products $xy$ in $\cds{9}$ by the author has never produced a value of $\displaystyle \frac{\norm{xy}}{\norm{x}\norm{y}}$ larger than $\sqrt{2}$. It is tempting to conjecture that $\norm{xy}\le\sqrt{2}\norm{x}\norm{y}$ for all Cayley-Dickson algebras. If the conjecture is true, then the preceding question may be answered in the affirmative.

\begin{question} Is there a Cayley-Dickson algebra and elements $x$ and $y$ such that $\norm{xy}>\sqrt{2}\norm{x}\norm{y}$ ?
\end{question}

In the Clifford algebras, the $n$-blades are distinct products of $n$ 1-blades and the properties of the basis vectors derive from the properties of 1-blades. A similar blade construction can be specified for the basis vectors of Cayley-Dickson algebras.

Cayley-Dickson basis vectors, or $n$-blades should be deducible from the properties of Cayley-Dickson 1-blades:
\begin{enumerate}
\item The square of 1-blades is $-1$. 
\item The product of 1-blades is anticommutative.
\item The conjugate of a 1-blade is its negative.
\item Every $n$-blade can be factored into the `left'-product of $n$ distinct 1-blades.
\end{enumerate}

An example of this last is $i_{51}=e_{1245}=e_1(e_2(e_4e_5)$. This works since, by induction, $\cyd(2^n,k2^{n+1})=1$ for all non-negative integers $n$ and positive integers $k.$ The Cayley-Dickson product is non-associative past the quaternions, so working out the correct sign for the product of Cayley-Dickson blades from their factors would be a challenge.

\begin{question}
 Can any insights into Cayley-Dickson algebras be gleaned from representing the basis vectors as products of 1-blades?
\end{question}

The Cayley-Dickson spaces are implicit in equations \ref{E:conj} and \ref{E:pairproduct}.

\begin{question}
 Do there exist equations corresponding to equations \ref{E:conj} and \ref{E:pairproduct} for the Clifford algebras?
\end{question}

The existence of the Fourier expansion (Theorem \ref{T:product}) and the adjoint properties of a twisted group product depend upon the twist being proper. This makes propriety an interesting twist property.
\begin{question}
 How many proper twists exist for a given group? Which of these have interesting algebras?
\end{question}

The proper twists on the groups $\mathbb{Z}_2,\mathbb{Z}_3.\mathbb{Z}_4$ and $\mathbb{Z}_5$ can be found by `back of the envelope' investigation, and they all are commutative and associative.

\begin{question}
 Are all proper twists on $\mathbb{Z}_N$ commutative and associative?
\end{question}

If $G$ is a group and if $\mathcal{P}(G)$ is the set of all proper twists on $G$, then $\mathcal{P}(G)$ is itself an abelian group with identity the \emph{trivial twist} $\iota(p,q)=1$ for all $p,q\in G$ and product $\alpha\beta(p,q)=\alpha(p,q)\beta(p,q)$ for $\alpha,\beta\in\mathcal{P}(G).$
\begin{question}
 How is the abelian group $\mathcal{P}(G)$ related to $G?$ Is there an interesting theory here?
\end{question}

\end{document}